\documentclass[final,nomarks]{dmtcs-episciences}
\usepackage[utf8]{inputenc}
\usepackage{subfigure}
\usepackage{latexsym}
\usepackage{amsfonts}
\usepackage{amsmath}
\usepackage{amssymb}
\usepackage{amsthm}
\usepackage{enumerate}
\usepackage{caption}
\usepackage{tabu}
\usepackage{float}
\usepackage{mathrsfs}

\newcommand{\Av}{\operatorname{Av}}

\usepackage{graphicx}
\usepackage{epstopdf}
\usepackage{epsfig}
\usepackage{caption}
 
\usepackage{bm} 
\usepackage{cite}

\makeatletter
\newtheorem*{rep@theorem}{\rep@title}
\newcommand{\newreptheorem}[2]{%
\newenvironment{rep#1}[1]{%
 \def\rep@title{#2 \ref{##1}}%
 \begin{rep@theorem}}%
 {\end{rep@theorem}}}
\makeatother

\newtheorem{theorem}{Theorem}[section]
\newreptheorem{theorem}{Conjecture}

\newtheorem{corollary}{Corollary}[section]

\theoremstyle{definition}

\newtheorem{remark}{Remark}[section]

\DeclareMathOperator{\Res}{Res}
\DeclareMathOperator{\inv}{inv}
\DeclareMathOperator{\EN}{EN}

\author{Colin Defant\affiliationmark{1,2}\thanks{The author was supported by a Fannie and John Hertz Foundation Fellowship and an NSF Graduate Research Fellowship.}}
\title[Pattern-Avoiding Linear Extensions]{Proofs of Conjectures about Pattern-Avoiding Linear Extensions}

\affiliation{
  Princeton University}
\keywords{Permutation pattern; linear extension; heap; rectangular poset.}

\received{2019-5-8}
\revised{2019-9-13}
\accepted{2019-9-15}

\begin{document}
\publicationdetails{21}{2019}{4}{16}{5438}
\maketitle

\begin{abstract}
After fixing a canonical ordering (or labeling) of the elements of a finite poset, one can associate each linear extension of the poset with a permutation. Some recent papers consider specific families of posets and ask how many linear extensions give rise to permutations that avoid certain patterns. We build off of two of these papers. We first consider pattern avoidance in $k$-ary heaps, where we obtain a general result that proves a conjecture of Levin, Pudwell, Riehl, and Sandberg in a special case. We then prove some conjectures that Anderson, Egge, Riehl, Ryan, Steinke, and Vaughan made about pattern-avoiding linear extensions of rectangular posets.  
\end{abstract}

\section{Introduction}\label{Sec:Intro}

Let $S_n$ be the set of permutations of $[n]=\{1,\ldots,n\}$, which we write as words in one-line notation. Given $\tau=\tau_1\cdots\tau_m\in S_m$, we say a permutation $\sigma=\sigma_1\cdots\sigma_n$ \emph{contains the pattern} $\tau$ if there exist indices $i_1<\cdots<i_m$ in $[n]$ such that for all $j,k\in[m]$, we have $\sigma_{i_j}<\sigma_{i_k}$ if and only if $\tau_j<\tau_k$. We say $\sigma$ \emph{avoids} $\tau$ if it does not contain $\tau$. Let $\Av_n(\tau^{(1)},\tau^{(2)},\ldots)$ be the set of permutations avoiding the patterns $\tau^{(1)},\tau^{(2)},\ldots$ (this list of patterns could be finite or infinite). 

The study of pattern avoidance in permutations, which began with Knuth's analysis of permutations that are sortable via a stack \cite{Knuth}, has grown into a large, thriving area of research in combinatorics \cite{Bona,Kitaev,Linton}. Recently, there has been a great amount of interest in pattern avoidance in other combinatorial objects such as words, inversion sequences, set partitions, and trees \cite{Albert, Anderson, Baxter, Bevan, Branden, Burstein2, Corteel, Dairyko, Daly, DefantEnumeration, DefantFertilityWilf, DefantPoset, DefantClass, DefantKravitz, DefantSupertrees, Godbole, Levin, Mansour1, Mansour2, Martinez, Rowland, Yakoubov}. The papers \cite{Anderson2, Bevan, DefantPoset, Levin, Yakoubov} consider the following general type of problem. Let $\mathcal P$ be an $n$-element poset, and suppose we are given some canonical total ordering of the elements of $\mathcal P$. We can view each linear extension of $\mathcal P$ as a bijective labeling of the elements of $\mathcal P$ with the elements of $[n]$. If we read these labels in the canonical order, we obtain a permutation. We can then ask how many linear extensions give rise to permutations that avoid certain patterns. An alternative, yet essentially equivalent, formulation of this problem is to consider a canonical \emph{labeling} of $\mathcal P$ and then view each linear extension of $\mathcal P$ as an ordering. Again, each linear extension gives rise to a permutation, so we can ask the same enumerative questions about pattern avoidance. 

As a first example, we consider the problems introduced in \cite{Levin}. A \emph{complete $k$-ary tree} is a rooted tree in which each vertex that is not in the penultimate or the last level has exactly $k$ children and all vertices in the last level are as far left as possible. We make the convention that trees grow up from their roots (as in real life). A \emph{$k$-ary heap} is a complete $k$-ary tree whose vertices are bijectively labeled with the elements of $[n]$ so that each vertex is given a label that is smaller than the labels of its children. Associating complete $k$-ary trees with posets in the natural way, one can view a $k$-ary heap as a linear extension of the underlying complete $k$-ary tree. We endow each complete $k$-ary tree with the breadth-first ordering (traversing the levels from bottom to top with each level traversed from left to right). Reading the labels in a $k$-ary heap in this order yields a permutation in $S_n$, which we call the permutation \emph{associated to} the $k$-ary heap. For example, the left image in Figure \ref{Fig1} shows the unique complete binary tree with $12$ vertices. The right image shows a binary heap obtained by labeling this tree. The permutation associated to this binary heap is $1\,3\,2\,4\,6\,9\,7\,12\,5\,11\,8\,10$. In Section \ref{Sec:Heaps}, we prove a conjecture from \cite{Levin} concerning binary heaps whose associated permutations avoid the pattern $321$. In fact, we will prove a much more general result. 

\begin{figure}[h]
\begin{center}
\includegraphics[width=.8\linewidth]{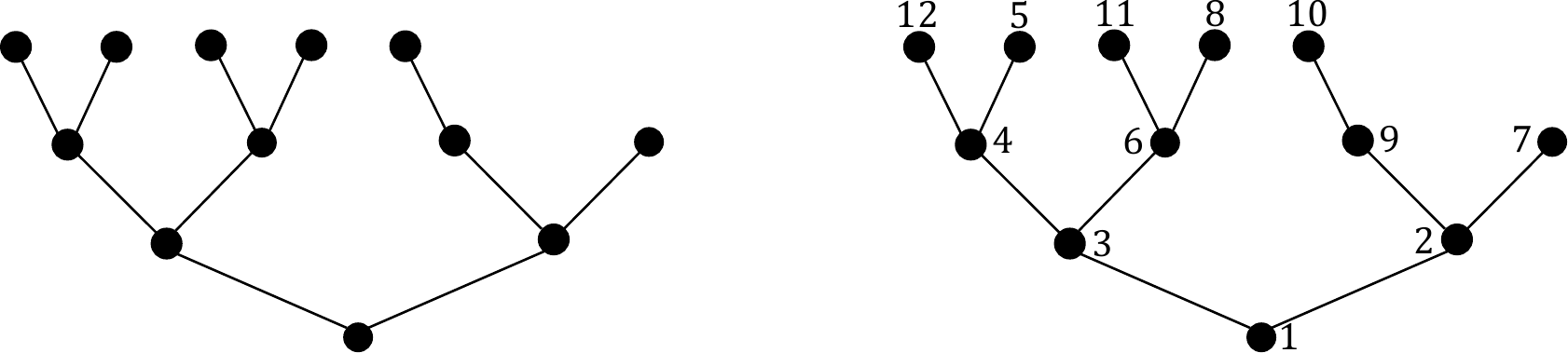}
\end{center}  
\caption{The complete binary tree with $12$ vertices (left) and a $12$-vertex binary heap (right). Reading the labels of this heap in the breadth-first order yields the associated permutation $1\,3\,2\,4\,6\,9\,7\,12\,5\,11\,8\,10$.}
\end{figure}\label{Fig1}

The second type of problem we consider deals with the rectangular posets $\EN_{s,t}$ studied in \cite{Anderson2}. The poset $\EN_{s,t}$ has $st$ elements that are labeled in a canonical fashion. The easiest way to define these posets is via examples. The Hasse diagrams of $\EN_{3,2}$, $\EN_{3,5}$, and $\EN_{4,3}$, along with their canonical labelings, are shown in Figure \ref{Fig2}. Each linear extension of $\EN_{s,t}$ can be viewed as an ordering of the labels, which is a permutation in $S_{st}$. For example, the linear extensions of $\EN_{3,2}$ correspond to the permutations $531642,536142,536412,563142$, and $563412$. In Section \ref{Sec:Rectangular}, we prove several of the conjectures that the authors of \cite{Anderson2} posed about pattern-avoiding linear extensions of these posets. 

\begin{figure}[h]
\begin{center}
\includegraphics[width=.65\linewidth]{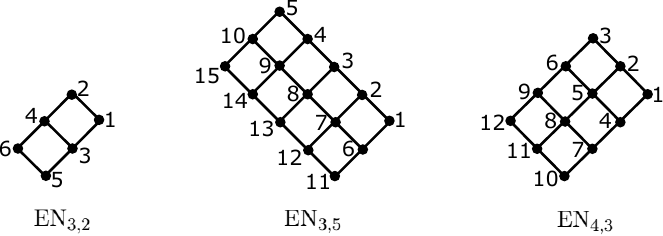}
\end{center}  
\caption{The Hasse diagrams of three rectangular posets. The numbers show the canonical labelings.}
\end{figure}\label{Fig2}

\section{Pattern-Avoiding $k$-ary Heaps}\label{Sec:Heaps}

Let $\mathcal H_n^k(\tau^{(1)},\tau^{(2)},\ldots)$ denote the set of $n$-vertex $k$-ary heaps whose associated permutations avoid \linebreak the patterns $\tau^{(1)},\tau^{(2)},\ldots$. The authors of \cite{Levin} found formulas for $|\mathcal H_n^k(\tau^{(1)},\ldots,\tau^{(r)})|$ for every set $\{\tau^{(1)},\ldots,\tau^{(r)}\}\subseteq S_3$ except the singleton set $\{321\}$. In fact, they were unable to explicitly enumerate binary heaps avoiding $321$. They did, however, compute $|\mathcal H_n^2(321)|$ for $1\leq n\leq 31$ and prove that \[2^{n-1}<|\mathcal H_n^2(321)|<4^n.\] Their data led them to conjecture that\footnote{The actual conjecture was written as ``$|\mathcal H_n^2(321)|\sim c^n$ for some $c\in(3.66,4)$," but it is clear from context that the authors really meant what is written in \eqref{Eq1}.} 
\begin{equation}\label{Eq1}
\lim_{n\to\infty}|\mathcal H_n^2(321)|^{1/n}\text{ exists and is in the interval }(3.66,4].
\end{equation} We will see that this conjecture follows as a special consequence of the main theorem of this section. To state this theorem, we need one additional piece of terminology. 

Given $\lambda=\lambda_1\cdots\lambda_\ell\in S_\ell$ and $\mu=\mu_1\cdots\mu_m\in S_m$, the \emph{direct sum} of $\lambda$ and $\mu$, denoted $\lambda\oplus\mu$, is the permutation in $S_{\ell+m}$ obtained by ``placing $\mu$ above and to the right of $\lambda$." More formally, the $i^\text{th}$ entry of $\lambda\oplus\mu$ is \[(\lambda\oplus\mu)_i=\begin{cases} \lambda_i & \mbox{if } 1\leq i\leq \ell; \\ \mu_{i-\ell}+\ell & \mbox{if } \ell+1\leq i\leq \ell+m. \end{cases}\] A permutation is called \emph{sum indecomposable} if it cannot be written as the direct sum of two shorter permutations. 

\begin{theorem}\label{Thm1}
Fix an integer $k\geq 2$ and a nonempty sequence of permutation patterns $\tau^{(1)},\tau^{(2)},\ldots$. If the permutations $\tau^{(1)},\tau^{(2)},\ldots$ are all sum indecomposable, then \[\lim_{n\to\infty}|\Av_n(\tau^{(1)},\tau^{(2)},\ldots)|^{1/n}=\lim_{n\to\infty}|\mathcal H_n^k(\tau^{(1)},\tau^{(2)},\ldots)|^{1/n}.\]
\end{theorem}

\begin{remark}\label{Rem1}
How do we know that the limits in Theorem \ref{Thm1} actually exist? Because the permutations $\tau^{(1)},\tau^{(2)},\ldots$ are sum indecomposable, we have an injective map $S_\ell\times S_m\hookrightarrow S_{\ell+m}$ given by $(\lambda,\mu)\mapsto\lambda\oplus\mu$ for all $\ell,m\geq 1$. This shows that the sequence $(|\Av_n(\tau^{(1)},\tau^{(2)},\ldots)|)_{n\geq 1}$ is supermultiplicative, so the limit on the left-hand side exists by Fekete's lemma. We also have an injective map $\mathcal H_n^k(\tau^{(1)},\tau^{(2)},\ldots)\hookrightarrow\Av_n(\tau^{(1)},\tau^{(2)},\ldots)$ sending each $k$-ary heap to its associated permutation. This shows that $\limsup\limits_{n\to\infty}|\mathcal H_n^k(\tau^{(1)},\tau^{(2)},\ldots)|^{1/n}\leq\lim\limits_{n\to\infty}|\Av_n(\tau^{(1)},\tau^{(2)},\ldots)|^{1/n}$. We will see in the proof of Theorem \ref{Thm1} that $\lim\limits_{n\to\infty}|\Av_n(\tau^{(1)},\tau^{(2)},\ldots)|^{1/n}\leq\liminf\limits_{n\to\infty}|\mathcal H_n^k(\tau^{(1)},\tau^{(2)},\ldots)|^{1/n}$, so the limit on the right-hand side exists as well. It is interesting to note that this right-hand limit does not depend on $k$. \hfill$\lozenge$ 
\end{remark}

Let $C_n=\frac{1}{n+1}{2n\choose n}$ denote the $n^\text{th}$ Catalan number. It is well known that $|\Av_n(321)|=C_n$. This implies that $\lim\limits_{n\to\infty}|\Av_n(321)|^{1/n}=4$. Consequently, we can specialize Theorem \ref{Thm1} to obtain the following result. In the special case in which $k=2$, this settles the conjecture from \cite{Levin} that is stated in \eqref{Eq1}. 

\begin{corollary}\label{Cor2}
If $k\geq 2$, then $\displaystyle\lim_{n\to\infty}|\mathcal H_n^k(321)|^{1/n}=4$.
\end{corollary}

\begin{proof}[of Theorem \ref{Thm1}]
Fix an integer $k\geq 2$ and a sequence of sum indecomposable permutation patterns $\tau^{(1)},\tau^{(2)},\ldots$. As mentioned in Remark \ref{Rem1}, we must prove that \[\lim_{n\to\infty}|\Av_n(\tau^{(1)},\tau^{(2)},\ldots)|^{1/n}\leq\liminf_{n\to\infty}|\mathcal H_n^k(\tau^{(1)},\tau^{(2)},\ldots)|^{1/n}.\] To ease 
notation, let \[a(n)=|\Av_n(\tau^{(1)},\tau^{(2)},\ldots)|^{1/n},\quad b(n)=|\mathcal H_n^k(\tau^{(1)},\tau^{(2)},\ldots)|^{1/n},\] \[\quad L_a=\lim_{n\to\infty}a(n),\quad L_b=\liminf_{n\to\infty}b(n).\] The proof is trivial if 
one of the patterns $\tau^{(i)}$ is the permutation $1\in S_1$, so we may assume otherwise. Because the patterns $\tau^{(1)},\tau^{(2)},\ldots$ are sum indecomposable, the identity permutation $123\cdots n$ is in $\Av_n(\tau^{(1)},\tau^{(2)},\ldots)$ and is the permutation associated to an element of $\mathcal H_n^k(\tau^{(1)},\tau^{(2)},\ldots)$. Thus, $L_a,L_b\geq 1$. Fix $\varepsilon\in(0,1)$, and let $N\geq 1$ be such that $a(n-\left\lfloor(n-1)/k\right\rfloor)>L_a-\varepsilon$ and $b(\left\lfloor(n-1)/k\right\rfloor)\!>L_b-\varepsilon$ for every $n\geq N$. Now choose $n\geq N$ such that $b(n)<L_b+\varepsilon$, and put $m=\left\lfloor(n-1)/k\right\rfloor$. 

\noindent {\bf Claim:} If $\lambda\in S_m$ is a permutation associated to a $k$-ary heap in $\mathcal H_m^k(\tau^{(1)},\tau^{(2)},\ldots)$ and \linebreak $\mu\in\Av_{n-m}(\tau^{(1)},\tau^{(2)},\ldots)$, then $\lambda\oplus\mu$ is a permutation associated to a $k$-ary heap in $\mathcal H_n^k(\tau^{(1)},\tau^{(2)},\ldots)$. 

Let us see how this claim implies Theorem \ref{Thm1}. Suppose by way of contradiction that $L_b<L_a$. Assuming the claim, we have \[(L_b+\varepsilon)^n>b(n)^n=|\mathcal H_n^k(\tau^{(1)},\tau^{(2)},\ldots)|\geq|\mathcal H_m^k(\tau^{(1)},\tau^{(2)},\ldots)|\cdot|\Av_{n-m}(\tau^{(1)},\tau^{(2)},\ldots)|\] \[=b(m)^ma(n-m)^{n-m}>(L_b-\varepsilon)^m(L_a-\varepsilon)^{n-m}.\] Since $m<n/2$ and $L_b<L_a$, we have $(L_b-\varepsilon)^m(L_a-\varepsilon)^{n-m}>(L_b-\varepsilon)^{n/2}(L_a-\varepsilon)^{n/2}$. Hence, $(L_b+\varepsilon)^2>(L_b-\varepsilon)(L_a-\varepsilon)$. Letting $\varepsilon$ tend to $0$ shows that $L_b\geq L_a$, contradicting our assumption that $L_b<L_a$. This completes the proof of Theorem \ref{Thm1} assuming the claim. 

Now, let $\lambda$ and $\mu$ be as in the claim. Let $T$ be the complete $k$-ary tree with $n$ vertices, and let $v_i$ be the vertex of $T$ that appears $i^\text{th}$ in the
breadth-first ordering. It is straightforward to check that the vertices $v_{m+1},\ldots,v_n$ are incomparable (meaning none is a descendant of another). It follows that $\lambda\oplus\mu$ is the permutation associated to some $k$-ary heap in $\mathcal H_n^k$. Because $\lambda$ and $\mu$ avoid the sum indecomposable permutations $\tau^{(1)},\tau^{(2)},\ldots$, their direct sum $\lambda\oplus\mu$ must also avoid these permutations. This completes the proof of the claim. 
\end{proof}

\section{Rectangular Posets}\label{Sec:Rectangular}

An \emph{inversion} of a permutation $\pi=\pi_1\cdots\pi_n\in S_n$ is a pair $(i,j)$ such that $i<j$ and $\pi_i>\pi_j$. Let $\inv(\pi)$ denote the number of inversions of $\pi$. Following \cite{Anderson2}, we let $\EN_{s,t}(\tau)(q)=\sum_{\pi\in\EN_{s,t}(\tau)}q^{\inv(\pi)}$, where $\EN_{s,t}(\tau)$ is the set of linear extensions of $\EN_{s,t}$ (viewed as permutations of the labels of $\EN_{s,t}$) that avoid the pattern $\tau$. The following three theorems were stated as conjectures in \cite{Anderson2}\footnote{Technically speaking, Theorem \ref{Thm2} was stated incorrectly in that article.}.
  
\begin{theorem}\label{Thm2}
For all $t\geq 1$, we have \[\EN_{3,t}(1243)(q)=\frac{q^{3(t^2-t+1)}(1-q^{2t-1}-2q^{2t}+q^{3t-1}+q^{3t})}{(1-q)(1-q^2)}.\]
\end{theorem}  

\begin{theorem}\label{Thm3}
For all $s\geq 1$, we have \[\EN_{s,2}(2143)(q)=q^{(2s-1)(s-1)}(1+q)^{s-1}.\]
\end{theorem}
  
\begin{theorem}\label{Thm4}
For all $s\geq 1$, we have \[\EN_{s,3}(2143)(q)=q^{9{s\choose 2}}F_s(1/q),\] where $F_s(r)$ is defined by $F_0(r)=F_1(r)=1$ and $F_s(r)=(1+r+2r^2)F_{s-1}(r)+r^3F_{s-2}(r)$ for $s\geq 2$. 
\end{theorem}

\begin{proof}[of Theorem \ref{Thm2}]
It is easy to verify this theorem when $t\in\{1,2\}$, so assume $t\geq 3$. Let $\pi=\pi_1\cdots\pi_n$ be a linear extension of $\EN_{3,t}(1243)$ (viewed as a permutation of the labels). Note that $2t$ appears before $2$ in $\pi$ because, otherwise, the entries $1,2,2t,t$ would form a $1243$ pattern. Similarly, $3t$ must appear before $t+2$, lest the entries $t+1,t+2,3t,2t$ form a $1243$ pattern. It follows that if we remove the entries $1$ and $t+1$ from $\pi$, then we will be left with the permutation \[(2t+1)(2t+2)\cdots(3t)(t+2)(t+3)\cdots(2t)23\cdots t.\] Let $i,j$ be such that $\pi_{i+1}=t+1$ and $\pi_{j+1}=1$. We can easily check that the possibilities for $i$ and $j$ are given by $i\in\{1,\ldots,t\}$ and $j\in\{i+1,\ldots,2t\}$. We find that \[\EN_{3,t}(1243)(q)=\sum_{i=1}^t\sum_{j=i+1}^{2t}q^{3t^2-3t+i+j}.\] This can easily be rewritten as $\dfrac{q^{3(t^2-t+1)}(1-q^{2t-1}-2q^{2t}+q^{3t-1}+q^{3t})}{(1-q)(1-q^2)}$.  
\end{proof}

\begin{proof}[of Theorem \ref{Thm3}]
Let $S$ be the collection of subsets of $\{1,3,5,\ldots,2s-3\}$, and define $\eta\colon\EN_{s,2}(2143)\to S$ by \[\eta(\pi)=\{i\in\{1,3,5,\ldots,2s-3\}\colon i+3\text{ appears before }i\text{ in }\pi\}.\] The map $\eta$ is a bijection. It is now easy to check that \[\EN_{s,2}(2143)(q)=\sum_{X\subseteq\{1,3,5,\ldots,2s-3\}}q^{(2s-1)(s-1)+|X|}=q^{(2s-1)(s-1)}(1+q)^{s-1}. \qedhere\] 
\end{proof}

\begin{proof}[of Theorem \ref{Thm4}]
Let $H_\ell$ be the set of $\pi=\pi_1\cdots\pi_{3\ell}\in\EN_{\ell,3}(2143)$ such that $\pi_2=3\ell-1$. Let $H_\ell(q)=\sum_{\pi\in H_\ell}q^{\inv(\pi)}$. Fix $s\geq 2$, and let $J(r_1,\ldots,r_k)$ be the set of permutations $\pi=\pi_1\cdots\pi_{3s}\in\EN_{s,3}(2143)$ such that $\pi_i=r_i$ for all $i\in\{1,\ldots,k\}$. One can check that the sets \[J(3s-2,3s-1,3s),\hspace{.5cm}J(3s-2,3s-1,3s-5,3s),\hspace{.5cm}J(3s-2,3s-5,3s-1,3s),\] \[J(3s-2,3s-5,3s-1,3s-4,3s),\hspace{.5cm}\text{and}\hspace{.5cm}J(3s-2,3s-1,3s-5,3s-4,3s)\] partition $\EN_{s,3}(2143)$. Call these sets $J_1,J_2,J_3,J_4$, and $J_5$, respectively. 

The operation that consists of removing the entries $3s-2,3s-1$, and $3s$ from a permutation establishes bijections $J_1\to\EN_{s-1,3}(2143)$, $J_2\to\EN_{s-1,3}(2143)$, $J_3\to\EN_{s-1,3}(2143)$, $J_4\to H_{s-1}$, and $J_5\to H_{s-1}$. After taking into account the number of inversions that are removed by each of these bijections, we obtain the identities \[\sum_{\pi\in J_1}q^{\text{inv}(\pi)}=q^{9(s-1)}\EN_{s-1,3}(2143)(q),\quad\sum_{\pi\in J_2}q^{\text{inv}(\pi)}=q^{9(s-1)-1}\EN_{s-1,3}(2143)(q),\] 
\[\sum_{\pi\in J_3}q^{\text{inv}(\pi)}=q^{9(s-1)-2}\EN_{s-1,3}(2143)(q),\quad\sum_{\pi\in J_4}q^{\text{inv}(\pi)}=q^{9(s-1)-3}H_{s-1}(q),\] and 
\[\sum_{\pi\in J_5}q^{\text{inv}(\pi)}=q^{9(s-1)-2}H_{s-1}(q).\] This yields 
\begin{equation}\label{Eq2}
\EN_{s,3}(2143)(q)=q^{9(s-1)-2}(1+q+q^2)\EN_{s-1,3}(2143)(q)+q^{9(s-1)-3}(1+q)H_{s-1}(q).
\end{equation}
The sets $J_1,J_2,J_5$ partition $H_s$, so 
\begin{equation}\label{Eq3}
H_s(q)=q^{9(s-1)-1}(1+q)\EN_{s-1,3}(2143)(q)+q^{9(s-1)-2}H_{s-1}(q).
\end{equation}
Solving the recurrence relations \eqref{Eq2} and \eqref{Eq3} subject to the initial conditions $\EN_{1,3}(2143)(q)=H_1(q)=1$, we obtain the desired identity $\EN_{s,3}(2143)(q)=q^{9{s\choose 2}}F_s(1/q)$. 
\end{proof}

The article \cite{Anderson2} also poses several conjectures concerning the polynomials $F_s(q)$ and various OEIS sequences \cite{OEIS}. We settle many of these conjectures\footnote{We also correct some typos made in the original statements of these conjectures.} in the following theorem. Since our focus is on the combinatorics of pattern-avoiding linear extensions and not these specific polynomials, we omit some details from the proof. In what follows, let $[q^r]G(q)$ denote the coefficient of $q^r$ in the Laurent series $G(q)$.

\begin{theorem}\label{Thm5}
Define the polynomials $F_s(r)$ by $F_0(r)=F_1(r)=1$ and $F_s(r)=(1+r+2r^2)F_{s-1}(r)+r^3F_{s-2}(r)$ for $s\geq 2$. For $s\geq 2$, 
\begin{itemize}
\item the values of $[q^3]F_s(q)$ are given by OEIS sequence A134465;
\item the values of $[q^{2s-2}]F_s(q)$ are given by OEIS sequence A098156; 
\item the values of $[q^{s-1}]F_s(q)$ are given by OEIS sequence A116914;  
\item the values of $[q^s]F_s(q)$ are given by OEIS sequence A072547;
\item the values of $[q^{s+1}]F_s(q)$ are given by OEIS sequence A002054;  
\item the values of $[q^{s+2}]F_s(q)$ are given by OEIS sequence A127531.
\end{itemize}
\end{theorem} 

\begin{proof}
Let $A(x,q)=\sum_{s\geq 0}F_s(q)x^s$. The recurrence for $F_s(q)$ translates into the identity 
\begin{equation}\label{Eq4}
A(x,q)=\frac{1-(q+2q^2)x}{1-(1+q+2q^2)x-q^3x^2}.
\end{equation}
Computing $\dfrac{1}{6}\dfrac{\partial^3}{\partial q^3}A(x,q)$ proves the first bullet point. The remainder of the proof makes use of the method of diagonals, which is discussed in Section 6.3 of \cite{Stanley}.

If we view $qA(x/q^2,q)$ as a function of the complex variable $q$, then \[\sum_{s\geq 0}([q^{2s-2}]F_s(q))x^s=[q^{-1}](qA(x/q^2,q))=\frac{1}{2\pi i}\int_{|q|=\rho}qA(x/q^2,q)\,dq,\] where $\rho>0$ is sufficiently small and the integral is taken over the circle of radius $\rho$ centered at the origin. By the Residue Theorem, this is \[\sum_{j=1}^r\Res_{q=u_j(x)}(qA(x/q^2,q)),\] where $u_1(x),\ldots,u_r(x)$ are the singularities of $qA(x/q^2,q)$ (viewed as functions of $x$) that tend to $0$ as $x\to 0$. We can explicitly compute that $r=2$ and that \[u_1(x)=\frac{x(1+x)+(1-x)\sqrt{x(4+x)}}{2(1-2x)}\quad\text{and}\quad u_2(x)=\frac{x(1+x)-(1-x)\sqrt{x(4+x)}}{2(1-2x)}.\] Let $U(x,q)=q^2(x-q(1-2x))$ and $V(x,q)=x+qx(1+x)-q^2(1-2x)$ so that $qA(x/q^2,q)=U(x,q)/V(x,q)$. Let $V_q(x,q)=\dfrac{\partial}{\partial q}V(x,q)$. Since $u_1(x)$ and $u_2(x)$ are simple poles of $qA(x/q^2,q)$, we find that \[\Res_{q=u_j(x)}(qA(x/q^2,q))=\frac{U(x,u_j(x))}{V_q(x,u_j(x))}\quad\text{for }j\in\{1,2\}.\] We have \[\sum_{j=1}^2\Res_{q=u_j(x)}(qA(x/q^2,q))=\sum_{j=1}^2\frac{U(x,u_j(x))}{V_q(x,u_j(x))}=\frac{x(1-2x+x^2+x^3)}{(1-2x)^2},\] and this proves the second bullet point. 

To prove the third, fourth, fifth, and sixth bullet points, we choose an integer $\ell\leq 2$ and view \linebreak $q^{-\ell-1}A(x/q,q)$ as a complex function of the variable $q$. As above, we have \[\sum_{s\geq 0}([q^{s+\ell}]F_s(q))x^s=[q^{-1}](q^{-\ell-1}A(x/q,q))=\sum_{j=1}^t\Res_{q=v_j(x)}(q^{-\ell-1}A(x/q,q)),\] where $v_1(x),\ldots,v_t(x)$ are the singularities of $q^{-\ell-1}A(x/q,q)$ that tend to $0$ as $x\to 0$. Let $Y(x,q)=-1+x+2qx$ and $Z(x,q)=x-(1-x)q+x(2+x)q^2$ so that $q^{-\ell-1}A(x/q,q)=q^{-\ell}Y(x,q)/Z(x,q)$. If $\ell\leq 0$, then the only singularity of $q^{-\ell-1}A(x/q,q)$ that tends to $0$ as $x\to 0$ is \[v_1(x)=\dfrac{1-x-\sqrt{1-2x-7x^2-4x^3}}{2x(2+x)}.\] If $\ell\in\{1,2\}$, then there is one other singularity, which is $v_2(x)=0$. One can check that \[\Res_{q=0}(q^{-2}A(x/q,q))=1-x^{-1}\quad\text{\and}\quad\Res_{q=0}(q^{-3}A(x/q,q))=1+2x^{-1}-x^{-2}.\] Note that in each of these expressions, the coefficient of $x^s$ is $0$ for every $s\geq 2$. It follows that for all integers $\ell\leq 2$ and $s\geq 2$, the coefficient of $x^s$ in $\sum_{s\geq 0}([q^{s+\ell}]F_s(q))x^s$ agrees with the coefficient of $x^s$ in $\Res_{q=v_1(x)}(q^{-\ell-1}A(x/q,q))$. We have \[\Res_{q=v_1(x)}(q^{-\ell-1}A(x/q,q))=\frac{v_1(x)^{-\ell}Y(x,v_1(x))}{Z_q(x,v_1(x))},\] where $Z_q(x,q)=\dfrac{\partial}{\partial q}Z(x,q)$. When $\ell\in\{-1,0,1,2\}$, we can explicitly compute and simplify these expressions in order to obtain proofs of the last four bullet points. 
\end{proof}

\section{Acknowledgments}
The author thanks the anonymous referees for helpful comments.

\end{document}